\documentclass[12pt]{amsart}
\usepackage{amscd,amsmath,amsthm,amssymb}
\usepackage{amsfonts,amssymb,amscd,amsmath,enumerate,verbatim}
\usepackage[left]{lineno}
\usepackage{pstricks, pst-plot,pst-3d}
\usepackage{tikz}
\newpsstyle{fatline}{linewidth=1.5pt}
\definecolor{verylight}{gray}{0.97}
\definecolor{light}{gray}{0.9}
\definecolor{medium}{gray}{0.85}
\definecolor{dark}{gray}{0.6}
\def\Jc{{\mathcal J}}
\def\Rc{{\mathcal R}}

\def\Gc{{\mathcal G}}

\def\Gc{{\mathcal G}}

\def\opn#1#2{\def#1{\operatorname{#2}}} 
\opn\chara{char} \opn\length{\ell} \opn\pd{pd} \opn\rk{rk}
\opn\projdim{proj\,dim} \opn\injdim{inj\,dim} \opn\rank{rank}
\opn\depth{depth} \opn\grade{grade} \opn\height{height}
\opn\embdim{emb\,dim} \opn\codim{codim}
\opn\Cl{Cl}

\opn\Tr{Tr} \opn\bigrank{big\,rank}
\opn\superheight{superheight}\opn\lcm{lcm}
\opn\trdeg{tr\,deg}
	\opn\reg{reg} \opn\lreg{lreg} \opn\ini{in} \opn\lpd{lpd}
	\opn\size{size} \opn\sdepth{sdepth}
	\opn\link{link}\opn\fdepth{fdepth}\opn\lex{lex}
	\opn\tr{tr}\opn\del{del}
	\opn\type{type}
	\opn\gap{gap}
	\opn\arithdeg{arith-deg}
	\opn\revlex{revlex}
	\opn\div{div} \opn\Div{Div} \opn\cl{cl} \opn\Cl{Cl}
	\opn\Spec{Spec} \opn\Supp{Supp} \opn\supp{supp} \opn\Sing{Sing}
	\opn\Ass{Ass} \opn\Min{Min}\opn\Mon{Mon}
	\opn\Ann{Ann} \opn\Rad{Rad} \opn\Soc{Soc}
	\opn\Im{Im} \opn\Ker{Ker} \opn\Coker{Coker} \opn\Am{Am}
	\opn\Hom{Hom} \opn\Tor{Tor} \opn\Ext{Ext} \opn\End{End}
	\opn\Aut{Aut} \opn\id{id}
	
	\opn\nat{nat}
	\opn\pff{pf}
	\opn\Pf{Pf} \opn\GL{GL} \opn\SL{SL} \opn\mod{mod} \opn\ord{ord}
	\opn\Gin{Gin} \opn\Hilb{Hilb}\opn\sort{sort}
	\opn\PF{PF}\opn\Ap{Ap}
	\opn\mult{mult}
	\opn\bight{bight}
	\opn\div{div}
	\opn\Div{Div}
	\opn\aff{aff}
	\opn\relint{relint} \opn\st{st}
	\opn\lk{lk} \opn\cn{cn} \opn\core{core} \opn\vol{vol}  \opn\inp{inp} \opn\nilpot{nilpot}
	\opn\link{link} \opn\star{star}\opn\lex{lex}\opn\set{set}
	\opn\width{wd}
	\opn\Fr{F}
	\opn\QF{QF}
	\opn\G{G}
	\opn\type{type}\opn\res{res}
	\opn\conv{conv}
	\opn\Deg{Deg}
	\opn\Sym{Sym}
	\opn\Con{Con}
	\opn\gr{gr}
	
	\def\pot#1#2{#1[\kern-0.28ex[#2]\kern-0.28ex]}

	\opn\dirlim{\underrightarrow{\lim}}
	\opn\inivlim{\underleftarrow{\lim}}
	\let\to=\rightarrow
	
	\def\Implies{\ifmmode\Longrightarrow \else
		\unskip${}\Longrightarrow{}$\ignorespaces\fi}
	\def\implies{\ifmmode\Rightarrow \else
		\unskip${}\Rightarrow{}$\ignorespaces\fi}
	\def\iff{\ifmmode\Longleftrightarrow \else
		\unskip${}\Longleftrightarrow{}$\ignorespaces\fi}

	\let\:=\colon
	\newtheorem{Theorem}{Theorem}[section]
	\newtheorem{Lemma}[Theorem]{Lemma}
	\newtheorem{Corollary}[Theorem]{Corollary}

	\newtheorem{Example}[Theorem]{Example}
	
	\newtheorem{Definition}[Theorem]{Definition}
	\newtheorem{Problem}[Theorem]{Problem}
	\newtheorem{Conjecture}[Theorem]{Conjecture}
	
	\newtheorem{Question}[Theorem]{Question}
	\let\epsilon\varepsilon
	\let\kappa=\varkappa
	\def\qed{\ifhmode\textqed\fi
		\ifmmode\ifinner\quad\qedsymbol\else\dispqed\fi\fi}
	\def\textqed{\unskip\nobreak\penalty50
		\hskip2em\hbox{}\nobreak\hfil\qedsymbol
		\parfillskip=0pt \finalhyphendemerits=0}
	\def\dispqed{\rlap{\qquad\qedsymbol}}
	
	\opn\dis{dis}
	\def\pnt{{\raise0.5mm\hbox{\large\bf.}}}
	
	\opn\Lex{Lex}

\begin{document}
\title[Partially ordered sets of distributive type]{Partially ordered sets of distributive type and algebras with straightening laws}

\author[T.~Hibi]{Takayuki Hibi}
\author[S.~A.~ Seyed Fakhari]{Seyed Amin Seyed Fakhari}

\address{(Takayuki Hibi) Department of Pure and Applied Mathematics, Graduate School of Information Science and Technology, Osaka University, Suita, Osaka 565--0871, Japan}
\email{hibi@math.sci.osaka-u.ac.jp}
\address{(Seyed Amin Seyed Fakhari) Departamento de Matem\'aticas, Universidad de los Andes, Bogot\'a, Colombia}
\email{s.seyedfakhari@uniandes.edu.co}

\subjclass[2020]{05E40, 13H10, 06D05}

\keywords{finite lattice, simplicial complex, partially ordered set of distributive type, algebra with straightening laws, Cohen--Macaulay partially ordered set}

\begin{abstract}
A finite poset (partially ordered set) $P$ with ${\hat 0}$ is called of distributive type if every interval $[{\hat 0}, a]$, $a \in P$, of $P$ is a distributive lattice.  From a viewpoint of ASL's (algebras with straightening laws), the join-meet toric ring on a finite distributive lattice is generalized to an ASL on a finite poset of distributive type.  Our target is the questions when a finite poset of distributive lattice is Cohen--Macaulay and when the ASL on it is Gorenstein.  We focus on a natural class of finite posets of distributive type and study various aspects of the above questions.
\end{abstract}	
\maketitle
\thispagestyle{empty}

\section*{Introduction}
A partially ordered set is called a poset.  Every poset in the present paper is finite.  A {\em poset of distributive type} is a poset $P$ with a unique minimal element ${\hat 0}$ for which every intertval $[{\hat 0}, a]$, $a \in P$, of $P$ is a distributive lattice.  A poset of distributive type is a generalization of a simplicial poset \cite{Sta91}.  Stanley \cite[Lemma 3.4]{Sta91} introduces an algebra with straightening laws \cite{Eis} on a simplicial poset.  Stanley's construction can be valid without modification for a poset of distributive type (Lemma \ref{L-I:ASL}).  The algebra with straightening laws on a poset of distributive type (Definition \ref{DEF}) is a generalization of that on a distributive lattice introduced in \cite{Hibi}.

Recall that a {\em dual order ideal} \cite[p.~246]{EC1} of a poset $P$ is a subset $I \subset P$ for which if $t \in I$ and $s \geq t$, then $s \in I$.  If $I$ is a dual order ideal of a distributive lattice $L$, then $L \setminus I$ is a poset of distributive type.  In particular, if $a \in L$, then $\{x \in L : x \geq a \}$ is a dual order ideal of $L$ and $L_a:= \{x \in L : x \not\geq a \}$ is a poset of distributive type.

Historically, when the research area {\sf Combinatorics and Commutative Algebra} was born in the middle of 1970s out of Stanley \cite{Sta75} and Reisner \cite{Rei}, the trends were Cohen--Macaulay and shellable partially ordered sets \cite{Bac, Bjo}.  A reasonable project is to find a natual class of Cohen--Macaulay (or shellable) posets of distributive type.

We focus on a poset of distributive type of the form $L \setminus I$, where $L$ is a  distributive lattice and $I$ is a dual order ideal of $L$.

\begin{Problem}
\label{L_a}
{\em
Classify the distributive lattices $L$ for which $L_a$ is Cohen--Macaulay for all $a \in L$.
}
\end{Problem}

In order to study more exciting problems, the class of $L \setminus I$ is rather large.  So, we pay our attention on the divisor lattice $D_{2^n\cdot 3^m}$.

\begin{Question}
    \label{CMshellable}
{\em
Let $L=D_{2^n\cdot 3^m}$ and $I$ a dual order ideal of $L$.  Is it true that $L \setminus I$ is Cohen--Macaulay if and only if $L \setminus I$ is shellable?
}
\end{Question}

\begin{Problem}
\label{L-I}
{\em
Find a combinatorial characterization of dual order ideals $I$ of the divisor lattice $L=D_{2^n\cdot 3^m}$ for which $L\setminus I$ is Cohen--Macaulay.
}
\end{Problem}

The reason why we are interested in a poset of distributive type is to find a possible generalization of the ring introduced in \cite{Hibi}.  Let $P$ be a poset of distributive type and $K[\{x_\alpha\}_{\alpha \in P}]$ the polynomial ring in $|P|$ variables over $K$.  We  introduce a polynomial $f_{\alpha,\beta}$, where $\alpha, \beta \in P$, as follows:
\begin{itemize}
\item
$f_{\alpha,\beta}=x_{\alpha}x_{\beta}$, if $\alpha$ and $\beta$ have no common upper bound in $P$;
\item
$f_{\alpha,\beta} = x_\alpha x_\beta - x_{\alpha \wedge \beta}(\sum_\gamma x_\gamma)$, where $\gamma$ ranges over all {\em minimal} upper bounds of $\alpha$ and $\beta$, otherwise.
\end{itemize}
We introduce the ideal
\[
J_{P}=(f_{\alpha,\beta}:\alpha, \beta \in P) \subset K[\{x_\alpha\}_{\alpha \in P}].
\]
and its quotient ring $\Rc_K[P]:=K[\{x_\alpha\}_{\alpha \in P}]/J_{P}$.  It follows that 
$\Rc_K[P]$ is Cohen--Macaulay if and only if $P$ is Cohen--Macaulay (Corollary \ref{L-I:CM:ASL}).

\begin{Question}
\label{linear_resolution}
{\em
When does $J_{P}$ have a linear resolution?
}
\end{Question}

A dual order ideal $I$ of a finite distributive lattice $L$ is called {\em rank-fixed} if the set of minimal elements of $I$ is the set of elements of a fixed rank of $L$.  It follows from \cite[Theorem 6.4]{Bac} that if a dual order ideal $I$ of $L$ is rank-fixed, then $L\setminus I$ is Cohen--Macaulay.  In particular, $\Rc_K[L\setminus I]$ is Cohen--Macaulay.

\begin{Question}
\label{rank-selected}
{\em
Let $L$ be a finite distributive lattice and $I$ a rank-fixed dual order ideal of $L$.  When is $\Rc_K[L\setminus I]$ Gorenstein?
}
\end{Question}

In the present paper, after summarizing fundamental materials on finite partially ordered sets together with simplicial complexes in Section $1$, in Section $2$, in order to explain our motivation why we are interested in a poset of distributive type, we recall what an algebra with straightening laws is from \cite{Eis} and generalize the ring introduced in \cite{Hibi}.  In Section $3$, Problem \ref{L_a} is solved (Theorem \ref{L_a:classification}).  In Section $4$, Question \ref{CMshellable} together with Problem \ref{L-I} is solved (Theorem \ref{divposet}).  Section $5$ is devoted to the study on Question \ref{linear_resolution}.  Finally, in Section $6$, partial answers to Question \ref{rank-selected} are given (Theorems \ref{boolgor}, \ref{rank-del} and \ref{level}).

\section{Cohen--Macaulay partially ordered sets}
We summarize fundamental materials on finite partially ordered sets together with simplicial complexes.  Standard references on the topic are \cite{HIBIred} and \cite{StaGREEN}.

A simplicial complex on the vertex set $V = \{x_1,x_2,\ldots, x_n\}$ is a collection $\Delta$ of subsets of $V$ for which (i) each $\{x_i\} \in \Delta$ and (ii) if $F \in \Delta$ and $F' \subset F$, then $F' \in \Delta$.  Each $F \in \Delta$ is called a {\em face} of $\Delta$.  A {\em facet} is a maximal face with respect to inclusion.  Set $d = \max\{|F| : F \in \Delta\}$ and define $\dim \Delta = d - 1$.  A simplicial complex $\Delta$ is {\em pure} if $|F| = d$ for all facets $F$ of $\Delta$.
Let $f(\Delta)=(f_0,f_1,\ldots, f_{d-1})$ be the {\em $f$-vector} of $\Delta$, where $f_i$ is the number of $F \in \Delta$ with $|F|=i+1$.  Letting $f_{-1} = 1$, define the {\em $h$-vector} $h(\Delta)=(h_0,h_1,\ldots,h_d)$ of $\Delta$ by the formula
$
\sum_{i=0}^{d} f_{i-1}(x-1)^{d-i}= \sum_{i=0}^{d} h_ix^{d-i}.
$
The reduced Euler characteristic of $\Delta$ will be denoted by $\widetilde{\chi}(\Delta)$.

Let $S=K[x_1, \ldots, x_n]$ denote the polynomial ring in $n$ variables over a field $K$ with each $\deg x_i = 1$.  If $F\subset V$, then set $x_F := \prod_{x_i \in F} x_i$.  The {\em Stanley--Reisner ideal} of $\Delta$ is the ideal $I_\Delta$ of $S$ which is generated by those squarefree monomials $x_\sigma$ with $\sigma \not\in \Delta$.  The {\em Stanley--Reisner ring} of $\Delta$ is the quotient ring $K[\Delta] = S/ I_\Delta$.
We say that $\Delta$ is {\em Cohen--Macaulay} (over $K$) if $K[\Delta]$ is Cohen--Macaulay.
The {\em link} of $F\in \Delta$ is the subcomplex $\link_{\Delta}F := \{F' \in \Delta : F \cap F' = \emptyset, F \cup F' \in \Delta\}$ of $\Delta$.  Every link of a Cohen--Macaulay complex is Cohen--Macaulay.
Every Cohen--Macaulay complex is pure.  Reisner \cite{Rei} gave a topological criterion for $K[\Delta]$ to be Cohen--Macaulay.
Stanley \cite{Sta75} used Reisner's criterion in order to prove affirmatively the upper bound conjecture for spheres.  Quick references on the background are \cite{HIBIred} and \cite{StaGREEN}.

A pure simplicial complex $\Delta$ of dimension $d-1$ is called {\em shellable} if there is an ordering $F_1, \ldots, F_s$ of the facets of $\Delta$ for which $(\bigcup_{j=0}^{i-1}\langle F_j\rangle) \cap \langle F_i\rangle$ is pure of dimension $d-2$ for each $1 < i \leq s$, where $\langle F_i\rangle = \{ \sigma \in \Delta : \sigma \subset F_i \}$.   Every shellable complex is Cohen--Macaulay.  A pure simplicial complex $\Delta$ is called {\em vertex decomposable} if either $\Delta$ is a simplex or there is a vertex $x$ of $\Delta$ for which
\begin{itemize}
    \item[(i)]
$\Delta - x := \{\sigma \in \Delta : x \not\in \sigma \}$ and $\link_{\Delta}\{x\}$ are vertex decomposable and
    \item[(ii)]
    no face of $\link_{\Delta}\{x\}$ is a facet of $\Delta - x$.
\end{itemize}
Every vertex decomposable complex is shellable.

A {\em chain} of a poset $P$ is a totally ordered subset of $P$. The {\em rank} of a poset $P$ is $\rank(P) = d-1$, where $d$ is the biggest cardinality of chains of $P$.
The {\em rank} of $x \in P$ is the biggest $r$ for which there is a chain of $P$ of the form $x_0 < x_1 <\cdots < x_r=x$.  Let $\rank_P(x)$ denote the rank of $x \in P$.  A poset $P$ is called {\em pure} if all maximal chains of $P$ have the same cardinality.
An {\em interval} of a poset $P$ is a subposet $[a,b]:=\{x \in P : a \leq x \leq b\}$, where $a,b \in P$ with $a < b$.  We refer the reader to \cite[pp.~157--159]{HHgtm260} for fundamental materials on {\em distributive lattices} $L=\Jc(P)$, {\em boolean lattices} $B_n$ and {\em divisor lattices} $D_n$.
The {\em order complex} of a poset $P$ is the simplicial complex $\Delta(P)$ on $P$ whose faces are chains of $P$.  A poset $P$ is called {\em Cohen--Macaulay} if $\Delta(P)$ is Cohen--Macaulay.  Every interval of a Cohen--Macaulay poset is Cohen--Macaulay.

\section{Algebras with straightening laws}
Let $R = \bigoplus_{n=0}^{\infty} R_n$ be a noetherian graded algebra over $K$.  Let $P$ be a finite poset and suppose that an injection $\varphi: P \hookrightarrow \bigcup_{n=1}^{\infty} R_n$ for which the $K$-algebra $R$ is generated by $\varphi(P)$ over $K$ is given.  A {\em standard monomial} is a homogeneous element of $R$ of the form $\varphi(\gamma_1) \varphi(\gamma_2)\cdots \varphi(\gamma_n)$, where $\gamma_1 \leq \gamma_2 \leq \cdots \leq \gamma_n$ in $P$.  We call $R$ an {\em algebra with straightening laws} \cite{Eis} on $P$ over $K$ if the following conditions are satisfied:
\begin{itemize}
\item[]
(ASL\,-1)
The set of standard monomials is a basis of $R$ over $K$;
\item[]
(ASL\,-2)
If $\alpha$ and $\beta$ in $P$ are incomparable and if
\begin{eqnarray*}
\label{ASL}
\, \, \, \, \, \, \, \, \, \, \, \, \, \, \, \, \, \, \, \,
\varphi(\alpha)\varphi(\beta)
= \sum_{i} r_i\,\varphi(\gamma_{i_1})\varphi(\gamma_{i_2}) \cdots , \, \, \, 0 \neq r_i \in K, \, \, \, \gamma_{i_1}\leq \gamma_{i_2} \leq \cdots
\end{eqnarray*}
is the unique expression for $\varphi(\alpha)\varphi(\beta) \in R$ as a linear combination of distinct standard monomials guaranteed by (ASL\,-1), then $\gamma_{i_1} \leq \alpha, \beta$ for every $i$.
\end{itemize}
The right-hand side of the relation in (ASL\,-2) is allowed to be the empty sum $(=0)$.  We abbreviate an algebra with straightening laws as ASL.  The relations in (ASL\,-2) are called the {\em straightening relations} for $R$.

Let $S=K[x_{\alpha} : \alpha \in P]$ denote the polynomial ring in $|P|$ variables over $K$ and define the surjective ring homomorphism $\pi: S \to R$ by setting $\pi(x_\alpha) = \varphi(\alpha)$. The defining ideal $I_R$ of $R=\bigoplus_{n=0}^{\infty} R_n$ is the kernel $\ker(\pi)$ of $\pi$. If $\alpha,\beta \in P$ are incomparable, define
\[
f_{\alpha,\beta}
:= x_\alpha x_\beta - \sum_{i} r_i\,x_{\gamma_{i_1}}x_{\gamma_{i_2}} \cdots x_{\gamma_{i_{n_i}}}, \, \text{ with } 0 \neq r_i \in K, \, \, \, \gamma_{i_1}\leq \gamma_{i_2} \leq \cdots,
\]
arising from (ASL-2). Then $f_{\alpha,\beta} \in I_R$. Let $\Gc_R$ denote the set of those polynomials $f_{\alpha,\beta}$ for which  $\alpha$ and $\beta$ are incomparable in $P$. Let $<_{\mathrm{rev}}$ denote the reverse lexicographic order \cite[Example~2.1.2(b)]{HHgtm260} on $S$ induced by an ordering of the variables for which $x_{\alpha} <_{\mathrm{rev}} x_{\beta}$ if $\alpha < \beta$ in $P$. It follows from (ASL-1) that $\Gc_R$ is a Gröbner basis of $I_R$ with respect to $<_{\mathrm{rev}}$. In particular, $I_R$ is generated by $\Gc_R$.

\begin{Definition}
\label{def}
{\em
Let $L$ be a lattice and $S=K[\{x_\alpha\}_{\alpha \in L}]$ the polynomial ring in $|L|$ variables over a field $K$.  A quadratic binomial
\[
f_{\alpha,\beta}= x_\alpha x_\beta - x_{\alpha \vee \beta}x_{\alpha \wedge \beta}, \quad \alpha, \beta \in L
\]
is called a {\em join-meet binomial} of $L$.
}
\end{Definition}

Clearly, $f_{\alpha,\beta} \neq 0$ if and only if $x_\alpha$ and $x_\beta$ are incomparable in $L$.  The ideal $J_L=(f_{\alpha,\beta}:\alpha, \beta \in L)$ and its quotient ring $S/J_L$ were introduced in \cite{Hibi}.  By virtue of the classical results of Birkhoff \cite[Theorem 9.1.7]{HHgtm260} and Dedekind \cite[Chapter 3, Exercise 30]{EC1}, it is shown that $J_L$ is a prime ideal if and only if $L$ is a distributive lattice and that, when $L$ is a distributive lattice, $S/J_L$ is an
ASL on $L$ over $K$ and is normal and Cohen--Macaulay.

\begin{Definition}
\label{DEF}
{\em
Let $P$ be a poset of distributive type and $K[P]=K[\{x_\alpha\}_{\alpha \in P}]$ the polynomial ring in $|P|$ variables over $K$.  Imitating \cite{Sta91}, we introduce a  polynomial $f_{\alpha,\beta}$, where $\alpha, \beta \in P$, as follows:
\begin{itemize}
\item
$f_{\alpha,\beta}=x_{\alpha}x_{\beta}$, if $\alpha$ and $\beta$ have no common upper bound in $P$;
\item
$f_{\alpha,\beta} = x_\alpha x_\beta - x_{\alpha \wedge \beta}(\sum_\gamma x_\gamma)$, where $\gamma$ ranges over all {\em minimal} upper bounds of $\alpha$ and $\beta$, otherwise.
\end{itemize}
}
\end{Definition}

Clearly $x_{\alpha \wedge \beta}$ exists if $\alpha$ and $\beta$ have an upper bound $z$, since $[{\hat 0}_P, z]$ is a distributive  lattice. Furthermore,  $f_{\alpha,\beta} \neq 0$ if and only if $\alpha$ and $\beta$ are incomparable in $P$.

We introduce the ideal
\[
J_{P}=(f_{\alpha,\beta}:\alpha, \beta \in P) \subset K[\{x_\alpha\}_{\alpha \in P}]
\]
and its quotient ring
\[
\Rc_K[P] := K[\{x_\alpha\}_{\alpha \in P}]/J_{P}. \]
In particular, if $L$ is a distributive lattice, then $\Rc_K[L]$ coincides with $S/J_L$ in the previous paragraph.

Now, the proof of Stanley \cite[Lemma 3.4]{Sta91} is valid for the following lemma.

\begin{Lemma}
    \label{L-I:ASL}
    The quotient ring $\Rc_K[P]$ is an ASL on $P$ over $K$.
\end{Lemma}

It then follows from \cite[Corollary 4.2]{Eis} and \cite[Corollary 2.7]{CV} that

\begin{Corollary}
    \label{L-I:CM:ASL}
    The quotient ring $\Rc_K[P]$ is Cohen--Macaulay if and only if $P$ is Cohen--Macaulay.
\end{Corollary}

\begin{Example}
    {\em  Let $P$ be the poset of distributive type of Figure $1$.  One has
    \[
    f_{x_2,x_3}= x_2x_3 - x_0(x_6+x_7).
    \]
    }
\end{Example}

 \begin{figure}[h]
\begin{tikzpicture}[scale=1]
\node[draw,shape=circle] (x0) at (0,0) {};
\node[draw,shape=circle] (x1) at (-1,1) {};
\node[draw,shape=circle] (x2) at (1,1) {};
\node[draw,shape=circle] (x3) at (-2,2) {};
\node[draw,shape=circle] (x4) at (0,2) {};
\node[draw,shape=circle] (x5) at (2,2) {};
\node[draw,shape=circle] (x6) at (-2,3) {};
\node[draw,shape=circle] (x7) at (0,3) {};
\node[draw,shape=circle] (x8) at (1,3) {};
\draw(x0)node[below=2mm]{\large{$x_0$}}--(x1)node[below=2mm]{\large{$x_1$}}--(x3)node[below=2mm]{\large{$x_3$}}--(x6)node[above=2mm]{\large{$x_6$}}--(x4)node[below=2mm]{\large{$x_4$}}--(x8)node[above=2mm]{\large{$x_8$}}--(x5)node[below=2mm]{\large{$x_5$}}--(x2)node[below=2mm]{\large{$x_2$}}--(x0);
\draw(x3)--(x7)node[above=2mm]{\large{$x_7$}}--(x4);
\draw(x1)--(x4)--(x2);
\end{tikzpicture}
    \caption{A poset of distributive type}
\end{figure}
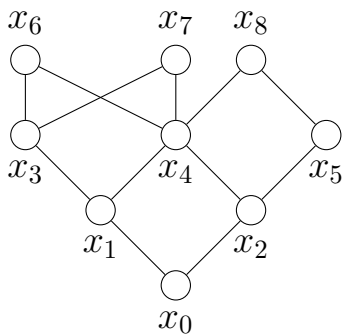

We close Section $2$ with a restriction lemma as follows.

\begin{Lemma}
\label{restrict}
Let $P$ be a poset of distributive type with the property that any pair of elements of $P$ has at most one minimal upper bound.  Let $L'=[a,b]$ be an interval of $P$ and set $S':=K[\{x_{\alpha}\}_{\alpha\in L'}]$.  Then, for all $i$ and $j$, one has
\[
\beta_{i,j}^{S'}(\Rc_K[L'])\leq \beta_{i,j}^S(\Rc_K[P]).
\]
\end{Lemma}

\begin{proof}
One easily sees that $J_{L'}=J_{P}\cap S'$, which shows that $\Rc_K[L']\subset \Rc_K[P]$ is an algebra retract. It follows from \cite[Corollary 2.5]{OHH} that $\beta_{i,j}^{S'}(\Rc_K[L'])\leq \beta_{i,j}^{S}(\Rc_K[P])$, as desired.
\hspace{12.3cm}
\end{proof}

\section{Problem \ref{L_a}}
The present section is devoted to solving Problem \ref{L_a}.  A classification of  distributive lattices $L$ for which $L_a$ is Cohen--Macaulay for all $a \in L$ is presented.

\begin{Lemma} \label{forbidden3}
Let $L=\Jc(P)$ be a distributive lattice and suppose that $P$ has three elements $p_1, p_2, q_1$ with $p_1<p_2$ for which $q_1$ is comparable with neither $p_1$ nor $p_2$.  Then there exists $a \in L$ for which $L_a$ is not Cohen--Macaulay.
\end{Lemma}

\begin{proof}
Let $a\in L$ denote the poset ideal of $P$ whose maximal elements are $p_1, q_1$ and $\Delta=\Delta(L_a)$ the order complex of $L_a$.  Our mission is to show that $\Delta$ is not Cohen--Macaulay.  Assume that $\Delta$ is Cohen--Macaulay.  Let $a_1\in L$ (resp. $a_2\in L$) denote the poset ideal of $P$ whose unique maximal element is $p_1$ (resp. $q_1$).  Set $a_3:=a_1\cap a_2$. Let $F\in \Delta$ denote the face consisting of a saturated chain of $L$ starting at $\emptyset$ and ending in $a_3$. Set $\Gamma:=\link_{\Delta} F$.

Let $b\in L$ denote the poset ideal of $L$ whose unique maximal element is $p_2$. It follows that $e=\{a_1, b\}$ is an edge of $\Gamma$ and $v:=\{a_2\}$ is a vertex of $\Gamma$. So, there are facets $G_1, G_2$ of $\Gamma$ with $e\subseteq G_1$ and $v\in G_2$. Since no element of $L_a$ (as a poset ideal of $P$) contains both $p_1, q_1$, we have $G_1\neq G_2$. Since $\Delta$ is Cohen--Macaulay, we deduce that $\Gamma $ is Cohen--Macaulay. In particular, there is a sequence $G_1=H_1, H_2, \ldots, H_s=G_2$ of facets of $\Gamma$ such that $|H_i\cap H_{i+1}|=|H_i|-1=|H_{i+1}|-1$ for each $1 \leq i\leq s$ (\cite[Lemma 9.1.12]{HHgtm260}). Recall that each $H_i$ is a chain of poset ideals of $P$.  Write $I_i$ for the maximal poset ideal of $H_i$ for each $1 \leq i \leq s$.

\medskip

{\bf Claim.} One has $I_i=I_{i+1}$ for each $1 \leq i \leq s-1$.

\medskip

\noindent
{\it Proof of the claim.} Suppose that Claim is false.  Let $i$ be the smallest integer for which $I_i\neq I_{i+1}$.  So, $I_1=\cdots = I_{i-i}=I_i$.  Since $H_j$ is a saturated chains of poset ideals of $P$, we assume that the chain of $H_j$ is given by poset ideals $$I_j^1\varsubsetneq I_j^2\varsubsetneq \ldots \varsubsetneq I_{j}^d=I_j.$$Since $I_i\neq I_{i+1}$ and $|H_i\cap H_{i+1}|=|H_i|-1=|H_{i+1}|-1$, we deduce that $I_i^k=I_{i+1}^k$ for each $1 \leq k \leq d-1$. Since the above chain is saturated, it follows that $$|I_i^d|=|I_i^{d-1}|+1=|I_{i+1}^{d-1}|+1=|I_{i+1}^d|.$$ Consequently, there are $x,y\in P$ for which $I_i^d=I_i^{d-1}\cup\{x\}$ and $I_{i+1}^d=I_i^{d-1}\cup\{y\}$.  First, assume that $x=p_1$ and $y=q_1$. Then $p_1\notin I_i^{d-1}$. Since $p_1<_Pp_2$, one has $p_2\notin I_i^{d-1}$.  Hence, $$p_2\notin I_i^{d-1}\cup\{p_1\}=I_i^{d-1}\cup\{x\}=I_i^d=I_i=I_1.$$This is a contradiction, as $p_2\in I_1$. Thus $(x,y)\neq (p_1,q_1)$. Furthermore, one has $(x,y)\neq (q_1,p_1)$, as $q_1\notin I_1$. As a result, $I_1^{d-1}\cup\{x,y\}$ is a poset ideal of $P$ which belongs to $L_a$. Therefore, the chain $$I_1^1\varsubsetneq I_1^2\varsubsetneq \ldots \varsubsetneq I_1^d\varsubsetneq I_1^{d-1}\cup\{x,y\}$$is a chain of poset ideals of $P$ which belong to $L_a$ and properly contains $H_1$, a contradiction.  This completes the proof of Claim.

\medskip

Now, Claim says that $I_1=I_s$. However, this is impossible, as $q_1\in I_s\setminus I_1$. \, \,
\end{proof}

\begin{Lemma} \label{graded}
Let $P$ be a poset and suppose that if $p_1 < p_2$ in $P$, then each $q \in P$ is comparable to either $p_1$ or $p_2$. Then $P$ is pure.
\end{Lemma}

\begin{proof}
Let $C_0: x_1<x_2<\ldots < x_m$ be a maximal chain of minimum length in $P$ and choose an arbitrary maximal chain $C: y_1<y_2<\ldots < y_n$
of $P$. Thus $m\leq n$.  Our mission is to show that $m=n$. Set$$\alpha(C):=m-\max\{i\mid x_k=y_k \ {\rm for} \ k=1, \ldots, i\}.$$Note that $\alpha(C)=m$ if $x_1\neq y_1$.  Our work proceeds with induction on $\alpha(C)$.  If $\alpha(C)=0$, then $C_0=C$ and there is nothing to prove. Suppose that $\alpha(C)\geq 1$. To simplify the notation, set$$\beta:=m-\alpha(C)=\max\{i\mid x_k=y_k \ {\rm for} \ k=1, \ldots, i\}.$$ Since $\alpha(C)\geq 1$, one has $\beta\leq m-1$. Furthermore, we know that $x_{\beta}=y_{\beta}$ and $x_{\beta+1}\neq y_{\beta+1}$.

Now, suppose that $m<n$ and consider the elements $x_{\beta+1}, y_{\beta+1}, y_{\beta+2}\in P$. If $x_{\beta+1}=y_{\beta+2}$, then $C_0\cup\{y_{\beta+1}\}$ is a chain of $P$ which properly contains $C_0$. This is a contradiction. Therefore, $x_{\beta+1}\neq y_{\beta+2}$. It follows from the assumption that $x_{\beta+1}$ is comparable to either $y_{\beta+1}$ or $y_{\beta+2}$. So, the following two cases arise.

\medskip

{\bf Case 1.} Let $x_{\beta+1}$ be comparable to $y_{\beta+1}$. If $x_{\beta+1}<y_{\beta+1}$, then $C\cup\{x_{\beta+1}\}$ is a chain of $P$ which properly contains $C$, a contradiction. Similarly, if $y_{\beta+1}<x_{\beta+1}$, then $C_0\cup\{y_{\beta+1}\}$ is a chain of $P$ which properly contains $C_0$, again a contradiction.

\medskip

{\bf Case 2.} Let $x_{\beta+1}$ be comparable to $y_{\beta+2}$. If $y_{\beta+2}<x_{\beta+1}$, then $C_0\cup\{y_{\beta+2}\}$ is a chain of $P$ which properly contains $C_0$, a contradiction. Hence, $x_{\beta+1}<y_{\beta+2}$. In this case, set $C':=(C\setminus\{y_{\beta+1}\})\cup\{x_{\beta+1}\}$. It follows that $C'$ is a chain of $P$ with the same length as $C$. Let $C''$ be a maximal chain of $P$ containing $C'$. Then $\alpha(C'')<\alpha(C)$. Thus we deduce from the induction hypothesis that $C_0$ and $C''$ have the same length. Since $|C_0|\leq |C|\leq |C''|$, it follows that $m=n$, as desired.
\,
\end{proof}

\begin{Lemma} \label{comprank}
Let $P$ be a pure poset satisfying the condition of Lemma \ref{graded}.  If $p,q \in P$ with $\rank_P(p) + 1 = \rank_P(q)$, then $p<q$.
\end{Lemma}

\begin{proof}
Suppose that  $p\not <_Pq$. Then $p$ and $q$ are not comparable in $P$. Since $P$ is pure, there is $p'\in P$ with $\rank_P(p')=\rank_P(p)+1$ with $p<p'$. It follows that $p'$ and $q$ are comparable in $P$, which is a contradiction, as $\rank_P(p')=\rank_P(q)$.
\end{proof}

\begin{Lemma} \label{vertdec}
Let $n\geq 2$ be an integer and $X$ a nonempty subset of $[n]=\{1,\ldots, n\}$.  Define $\Delta_{n,X}$ to be the simplicial complex whose facets are those $(n-1)$-subsets $W$ of $[n]$ with $X \not\subset W$.  Then $\Delta_{n,X}$ is vertex decomposable.
\end{Lemma}

\begin{proof}
We proceed by induction on $n$.  Our assertion is trivial for $n=2$. Let $n\geq 3$ and $x\in X$. We easily see that $\link_{\Delta_{n,X}} \{x\}=\Delta_{n-1,X\setminus\{x\}}$. Our induction hypothesis says that  $\link_{\Delta_{n,X}} \{x\}$ is vertex decomposable.  Since $\Delta_{n,X} -x$ is a simplex on $n-1$ vertices, it is vertex decomposable.  Furthermore, since each facet of $\Delta_{n,X} -x$ is a facet of $\Delta_{n,X}$, it follows that $\Delta_{n,X}$ is vertex decomposable, as desired.
\, \, \, \, \,
\end{proof}

\begin{Lemma} \label{vertdecbool}
Let $L$ be a boolean lattice and $a \in L$. Then $L_a$ is vertex decomposable.
\end{Lemma}

\begin{proof}
There is nothing to prove for $a=\hat{0}$. So, suppose that $a\neq \hat{0}$. Let $L=B_n$.  So, $a=X$ for a nonempty subset $X$ of $[n]$. The order complex of $L_a$ is the barycentric subdivision of the simplicial complex $\Delta_{n,X}$ as in Lemma \ref{vertdec}. The purity of the order complex of $L_a$ follows from the purity of $\Delta_{n,X}$ and its vertex decomposability is a consequence of \cite[Corollary 11.7]{BjWa} and Lemma \ref{vertdec}.
\, \, \, \,
\end{proof}

\begin{Lemma}
    \label{clutter}
    The following conditions on a poset $P$ are equivalent:
\begin{itemize}
\item [(i)] If $p_1<p_2$ in $P$, then each $q\in P$ is comparable to either $p_1$ or $p_2$.

\item [(ii)] $P$ is pure and satisfies that if $p,q \in P$ with ${\rm rank}_P(p)<{\rm rank}_P(q)$, then $p<q$.

\item [(iii)] $P$ is the sum \cite[p.~246]{EC1} of antichains \cite[p.~201]{HHgtm260}.
\end{itemize}
\end{Lemma}

\begin{proof} First, (i) $\Rightarrow$ (ii) follows from Lemmata \ref{graded} and \ref{comprank}.  Let $P$ be a poset which satisfies (ii) and $A_i$ the antichain of $P$ consisting of those $x \in P$ with $\rank_P(x) = i$.  If $x \in A_i$ and $y \in A_j$ with $i < j$, then $x < y$.  Thus $P$ is the sum of antichains $A_0, A_1, \ldots, A_{\rank(P)}$, as desired.  Finally, (iii) $\Rightarrow$ (i) is clear.
\, \, \, \, \, \, \, \, \, \, \, \, \,
\end{proof}

We now come to a complete solution of Problem \ref{L_a}.

\begin{Theorem}
\label{L_a:classification}
Let $L=\Jc(P)$ be a distributive lattice. Then the following conditions are equivalent:

\begin{itemize}
\item [(i)] $P$ is the sum of antichains.

\item [(ii)] $L_a$ is vertex decomposable for each $a\in L$.

\item [(iii)] $L_a$ is shellable for each $a\in L$.

\item [(iv)] $L_a$ is Cohen--Macaulay for each $a\in L$.
\end{itemize}
\end{Theorem}

\begin{proof}  First, (ii) $\Rightarrow$ (iii) $\Rightarrow$ (iv) are true in general.  Second, (iv) $\Rightarrow$ (i) follows from Lemmata \ref{forbidden3} and \ref{clutter}.  On the other hand, (i) $\Rightarrow$ (ii) follows from Lemma \ref{vertdecbool} together with the fact that the sum of (pure and) vertex decomposable posets is (pure and) vertex decomposable.
\, \, \, \, \, \, \, \, \, \, \, \, \, \, \, \, \, \, \, \, \, \, \, \, \, \, \, \, \, \, \, \, \, \, \, \, \, \, \, \,
\end{proof}

Let $\hat{1}$ denote the unique maximal element of a given lattice $L$. An {\em apex} of a lattice $L$ is $a \in L$ with $a \neq {\hat 0}$ and $a \neq {\hat 1}$ for which $a$ is comparable to all $x \in L$.  A lattice with no apex is called {\em simple}.

\begin{Corollary}
\label{L_a:classificationCOR}
Let $L$ be a simple distributive lattice. Then the following conditions are equivalent:

\begin{itemize}
\item [(i)] $L$ is boolean.

\item [(ii)] $L_a$ is vertex decomposable for each $a\in L$.

\item [(iii)] $L_a$ is shellable for each $a\in L$.

\item [(iv)] $L_a$ is Cohen--Macaulay for each $a\in L$.
\end{itemize}
\end{Corollary}

\section{Question \ref{CMshellable} and Problem \ref{L-I}}

We now turn to the discussion of Question \ref{CMshellable} and Problem \ref{L-I}.

\begin{Lemma}
\label{NEW1}
Let $L=D_{2^n\cdot 3^m}$ and $I$ a dual order ideal of $L$.  Suppose that
\[
2^{\alpha_1}3^{\beta_1}, \ldots, 2^{\alpha_k}3^{\beta_k}
\]
are the minimal elements of $I$, where
\[
\alpha_1> \cdots >\alpha_k, \quad \beta_1<\cdots <\beta_k.
\]

(a) If $1\leq \alpha_1\leq n-1$ and $1 \leq \beta_1$, then $L\setminus I$ is not Cohen--Macaulay.

(b) If $\alpha_1 = n$ and $1 \leq \beta_1$ and if $L\setminus I$ is Cohen--Macaulay, then
$I$ is rank-fixed.
\end{Lemma}

\begin{proof}
(a) The interval $L'=\{x \in L \setminus I : x \geq 2^{\alpha_1-1}3^{\beta_1-1}\}$ of $L \setminus I$ is the union of chains $C$ and $C'$ with $C \cap C' = 2^{\alpha_1-1}3^{\beta_1-1}$ and $|C| \geq 3$. Hence, $L'$ is not be Cohen--Macaulay.  It then follows that $L \setminus I$ is not Cohen--Macaulay, as desired.

(b) Since $L\setminus I$ is Cohen--Macaulay, the interval $L'=\{x \in L \setminus I : x \geq 2^{n-1}3^{\beta_1-1}\}$ of $L \setminus I$ is pure.  Thus $\alpha_2 = n-1$ and $\beta_2 = \beta_1+1$.  Repeating the procedure yields that $I$ is rank-fixed.
\hspace{10cm}
\end{proof}

It follows from Lemma \ref{NEW1} that the remaining issue is $2^{\alpha_1}3^{\beta_1}= 2^{n'}$ with $1 \leq n' \leq n$.  Furthermore, if $2^{\alpha_1}3^{\beta_1}= 2^{n'}$, then $L \setminus I = L'\setminus I'$, where $L'=D_{2^{n'-1}\cdot 3^m}$ and $I'$ is a dual order ideal of $L'$.  On the other hand, if $\alpha_1 =0$, then $L \setminus I = D_{2^{n}\cdot 3^{\beta_1-1}}$.

\begin{Corollary}
    \label{NEW2}
Let $L=D_{2^n\cdot 3^m}$ and $I$ a dual order ideal of $L$. Then the following conditions are equivalent:
\begin{itemize}
\item [(i)]
$L\setminus I$ is a sublattice of $L$ of the form $D_{2^{n'}\cdot 3^{m'}} \setminus I'$, where $n' \leq n, m' \leq m$ and where $I'$ is a rank-fixed dual order ideal of $D_{2^{n'}\cdot 3^{m'}}$.

\item [(ii)] $L \setminus I$ is Cohen--Macaulay.
\end{itemize}
\end{Corollary}

\begin{Lemma} \label{divvertdecom}
Let $L=D_{2^n\cdot 3^m}$ and $I$ a rank-fixed dual order ideal of $L$. Then $L\setminus I$ is vertex decomposable.
\end{Lemma}

\begin{proof}
We easily see the purity of $L\setminus I$. The vertex decomposability of $L \setminus I$ follows from \cite[Theorems 10.11 and 11.6]{BjWa} together with the fact that every distributive lattice is CL-shellable.
\hspace{7.3cm}
\end{proof}

We now come to a complete solutions of Question \ref{CMshellable} as well as Problem \ref{L-I}.

\begin{Theorem} \label{divposet}
Let $L=D_{2^n\cdot 3^m}$ and $I$ a dual order ideal of $L$. Then the following conditions are equivalent:
\begin{itemize}
\item [(i)] $L\setminus I$ is a sublattice of $L$ of the form $D_{2^{n'}\cdot 3^{m'}} \setminus I'$, where $n' \leq n, m' \leq m$ and where $I'$ is a rank-fixed dual order ideal of $D_{2^{n'}\cdot 3^{m'}}$.

\item [(ii)] $L\setminus I$ is vertex decomposable.

\item [(iii)] $L\setminus I$ is shellable.

\item [(iv)] $L\setminus I$ is Cohen--Macaulay.
\end{itemize}
\end{Theorem}

\begin{proof}
First, (i) $\Rightarrow$ (ii) follows from Lemma \ref{divvertdecom}.   Second, (ii) $\Rightarrow$ (iii) $\Rightarrow$ (iv) are true in general.  Finally, (iv) $\Rightarrow$ (i) follows from Corollary \ref{NEW2}.
 \, \, \, \, \, \, \, \, \, \,
 \, \, \, \, \,
\end {proof}

\section {Question \ref{linear_resolution}}
Let $P$ be a poset of distributive type and $K[P] = K[\{x_\alpha\}_{\alpha \in P}]$ the polynomial ring in $|P|$ variablrs over $K$.  Since the quotient ring $\Rc_K[P] = K[P]/J_{P}$ is an ASL on $P$ over $K$, it follows from the second paragraph of Section $2$ that
\[
\Gc_{P}:=\{f_{\alpha,\beta} :\alpha, \beta \in P\}
\]
is a Gr\"obner basis of $J_{P}$ with respect to the reverse lexicographc order $<_{\rm rev}$ on $K[P]$ induced by the ordering of the variables for which $x_{\alpha} < x_{\beta}$ if $\alpha < \beta$ in $P$.

Recall that the {\em comparablility graph} of a poset $P$ is a graph on $P$ whose edges are those $\{x,y\}$ with either $x < y$ or $y < x$.

\begin{Theorem}
\label{chordal}
The ideal $J_{P}$ has a linear resolution if and only if the comparability graph of $P$ is chordal.
\end{Theorem}

\begin{proof}
Since $J_{P}$ has a squarefree initial ideal, it follows from \cite[Corollary 2.7]{CV} that $J_{P}$ has a linear resolution if and only if its initial ideal has a linear resolution.  Since the initial ideal is the edge ideal of the complement of the comparability graph of $P$, it follows from Fr\"oberg's criterion \cite[Theorem 9.2.3]{HHgtm260} that the initial ideal has linear resolution if and only if the comparability graph of $P$ is chordal.
\, \, \, \, \, \, \, \, \,
\end{proof}

Corollary \ref{EHH} is shown in \cite[Theorem 4.2]{EHH} by means of Hilbert functions.

\begin{Corollary}
\label{EHH}
Let $L$ be a simple distributive lattice.  Then $J_{L}$ has a linear resolution if and only if $L = D_{2 \cdot 3^n}$ with $n \geq 1$.
\end{Corollary}

\begin{proof}
The comparability graph of $L = D_{2 \cdot 3^n}$ with $n \geq 1$ is the finite graph ${\rm Com}(L)$ on $\{1,\ldots,n, 1', \ldots, n'\}$ whose edges are $\{i,j\}, \{i',j'\}$ with $i < j$ and $\{i,j'\}$ with $j \leq i$.  Since the induced subgraph of ${\rm Com}(L)$ on $\{1,\ldots,n\}$ and that on $\{1', \ldots, n'\}$ are complete graphs, we easily see that ${\rm Com}(L)$ is a chordal graph.

If $L$ is nonplanar, then an interval of $L$ is the boolean lattice $B_3$.  Since the comparability graph of $B_3$ is not chordal, it follows that ${\rm Com}(L)$ cannot be chordal.  Suppose that $L = \Jc(P)$ is planar.  Let $C:x_0 < x_1 < \cdots < x_s$ and $C': y_0 < y_1 < \cdots < y_t$, where $s > 0$ and $t > 0$, be chains of $P$ with $P=C\cup C'$ and $C \cap C' = \emptyset$.  Since $L$ is simple, $x_0$ and $y_0$ (resp. $x_s$ and $y_t$) are incomparable.  Let $a = P \setminus \{x_s\}, b = P \setminus \{y_t\}, a'= \{x_0\}, b' = \{y_0\}$ be elements of $\Jc(P)$.  Then the induced subgraph of ${\rm Com}(L)$ on $\{a,b,a',b'\}$ is the cycle of length $4$ without any chord.  Thus ${\rm Com}(L)$ cannot be chordal, as desired.
\, \, \, \, \, \, \, \, \, \, \, \, \, \, \, \, \, \,
\end{proof}

\begin{Example}
{\em
Let $L$ be a distributive lattice and $I$ a dual order ideal. Assume that $J_{L\setminus I}$ has a linear resolution and suppose that $I'$ is a dual order ideal of $L$ with $I \subset I'$.  It then follows from Theorem \ref{chordal} that $J_{L\setminus I'}$ has a  linear resolution.  In particular, if $I$ is a dual order ideal of the divisor lattice $D_{2 \cdot 3^n}$ with $n \geq 1$, then $J_{D_{2 \cdot 3^n} \setminus I}$ has a linear resolution.
}
\end{Example}

\begin{Example}
{\em
Let $P$ be the poset of distributive type of Figure $2$.  Since the comparability graph of $P$ is chordal, the ideal $J_P$ has a linear resolution.
}
\end{Example}

 \begin{figure}[h]
\begin{tikzpicture}[scale=0.7]
\node[draw,shape=circle] (1) at (0,0) {};
\node[draw,shape=circle] (2) at (1,1) {};
\node[draw,shape=circle] (3) at (2,2) {};
\node[draw,shape=circle] (4) at (3,3) {};
\node[draw,shape=circle] (5) at (1,-1) {};
\node[draw,shape=circle] (6) at (2,0) {};
\node[draw,shape=circle] (7) at (3,1) {};
\node[draw,shape=circle] (8) at (4,2) {};
\node[draw,shape=circle] (9) at (2,-2) {};
\node[draw,shape=circle] (10) at (3,-1) {};
\node[draw,shape=circle] (11) at (4,0) {};
\node[draw,shape=circle] (12) at (3,-3) {};
\node[draw,shape=circle] (13) at (4,-2) {};
\node[draw,shape=circle] (14) at (5,-1) {};
\node[draw,shape=circle] (15) at (4,-4) {};
\node[draw,shape=circle] (16) at (5,-3) {};
\node[draw,shape=circle] (17) at (5,-5) {};
\node[draw,shape=circle] (18) at (6,-4) {};
\draw(1)--(2)--(3)--(4)--(8)--(7)--(6)--(5);
\draw(2)--(6);
\draw(3)--(7);
\draw(1)--(5)--(9)--(12)--(15)--(17)--(18)--(16)--(15);
\draw(9)--(10)--(11)--(14)--(13)--(12);
\draw(10)--(13);
\end{tikzpicture}
    \caption{A poset with linear resolution}
\end{figure}
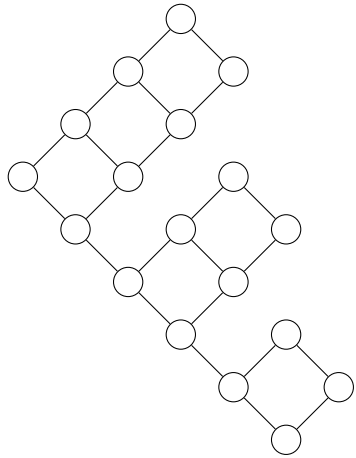

It would, of course, be of interest to find a reasonable class of posets $P$ of distributive type for which $J_P$ has linear resolution.

\section {Question \ref{rank-selected}}
In \cite[p.~105]{Hibi}, it is shown that $\Rc_K[L]$ is Gorenstein if and only if $P$ is pure, where $L=\Jc(P)$.  However, except for $I = \emptyset$, finding a complete solution to Question \ref{rank-selected} seems difficult.  We focus on Question \ref{rank-selected} for the boolean lattice $L = B_n$.

\begin{Lemma}
    \label{Boston}
    Let $I$ be a rank-fixed dual order ideal of $B_n$ and $d + 1$ the rank of the minimal elements of $I$.  If $1 \leq d \leq n-2$, then $\Rc_K[B_n\setminus I]$ cannot be Gorenstein.
\end{Lemma}

\begin{proof}
    Let $\Delta$ denote the simplicial complex on $[n]=\{1,\ldots,n\}$ which consists of $F \subset [n]$ with $|F| \leq d$.  It follows that the order complex $\Gamma$ of $(B_n \setminus I)\setminus \{{\hat 0}\}$ is the barycentric subdivision of $\Delta$.  Let $h(\Gamma)=(h_0,h_1,\ldots,h_d)$ denote the $h$-vector of $\Gamma$ and $h(\Delta)=(h'_0,h'_1,\ldots,h'_d)$ that of $\Delta$. One has
\[
(-1)^{d-1} h_d = \widetilde{\chi}(\Gamma)=\widetilde{\chi}(\Delta)= (-1)^{d-1}  h'_d
\]
and, by using the formula ${p \choose q}={p-1 \choose q-1}+{p-1 \choose q}$, it follows that
\[
(-1)^{d-1}h'_d = -{n \choose 0}+ {n \choose 1} -{n \choose 2} + \cdots + (-1)^{d-1}{n \choose d} = (-1)^{d-1}{n-1 \choose d}.
\]
Since $0 < d < n-1$, it follows that $h'_d \geq 2$.  Since $(h_0,h_1,\ldots,h_d)$ is the $h$-vector of $\Rc_K[B_n \setminus I]$, it follows that $\Rc_K[B_n \setminus I]$ cannot be Gorenstein.
\, \, \, \, \, \, \, \, \, \, \,
\end{proof}

\begin{Theorem} \label{boolgor}
Let $I$ be a rank-fixed dual order ideal of $B_n$.  Then $\Rc_K[B_n\setminus I]$ is Gorenstein if and only if $I=\emptyset$ or $I=\{{\hat 1}_{B_n}\}$ or $B_n \setminus I=\{{\hat 0}_{B_n}\}$ or $B_n \setminus I=\emptyset$.
\end{Theorem}

\begin{proof}
The ``Only If'' part follows from Lemma \ref{Boston}.  It follows from \cite[p.~105]{Hibi} that $\Rc_K[B_n]$ is Gorenstein.  Furthermore, since $\hat{1}_{B_n}$ is a nonzero divisor of $\Rc_K[B_n]$, $\Rc_K[B_n\setminus \{\hat{1}_{B_n}\}]$ is Gorenstein.  On the other hand, $\Rc_K[\{{\hat 0}_{B_n}\}]$ is the polynomial ring in one variable and $\Rc_K[\emptyset]=K$.
\, \, \, \, \, \, \, \, \, \, \, \, \, \, \, \, \, \, \, \, \, \, \, \, \, \, \, \, \, \, \, \,
\end{proof}

On the other hand, Lemma \ref{Boston} brings a result on regularity of $\Rc_K[B_n\setminus I]$.

\begin{Corollary}
    \label{regularity}
Let $I$ be a rank-fixed dual order ideal of the boolean lattice $B_n$ and $d + 1$ the rank of the minimal elements of $I$.  Suppose that $1 \leq d \leq n-2$.  Then
\[
\reg(\Rc_K[B_n\setminus I]) = \rank(B_n \setminus I).
\]
\end{Corollary}

\begin{proof}
Note that $\rank(B_n \setminus I) = d$. Let $(h_0,h_1,\ldots,h_d)$ be the $h$-vector of $\Rc_K[B_n\setminus I]$.  Since $B_n\setminus I$ is Cohen--Macaulay, it is explicitly stated in, e.g., \cite[Lemma 2.5]{BV} that
$$
\reg(\Rc_K[B_n\setminus I]) = \max\{i : h_i \neq 0\}.
$$
Now, by the proof of Lemma \ref{Boston}, one has $h_d \neq 0$ and
$$\reg(\Rc_K[B_n\setminus I]) = d = \rank(B_n \setminus I),$$
as desired.
\hspace{11.75cm}
\end{proof}

We demonstrate a limited partial solution to Question \ref{rank-selected}. We recall that for a given simplicial complex $\Delta$, its $i$th reduced homology with coefficients in $K$ is denoted by $\widetilde{H}_i(\Delta; K)$.

\begin{Theorem} \label{rank-del}
Let $L$ be a distributive lattice and $\rho_i$ the number of elements of $L$ of rank $i$.
Let $I$ be a rank-fixed dual order ideal of $L$.

\begin{itemize}
\item [(i)] If $\rank(L \setminus I) =0$, then  $\Rc_K[L\setminus I]$ is Gorenstein.

\item [(ii)] If $\rank(L \setminus I) =1$, then  $\Rc_K[L\setminus I]$ is Gorenstein if and only if $\rho_1\leq 2$.

\item [(iii)] If $\rank(L \setminus I) =2$, then $\Rc_K[L\setminus I]$ is Gorenstein if and only if $$(\rho_1, \rho_2)\in \{(1,1), (1,2), (2,1), (3,3)\}.$$
\end{itemize}
\end{Theorem}

\begin{proof}
(i) If $\rank(L \setminus I) =0$, then  $\Rc_K[L\setminus I]$ is the polynomial ring on one variable which is Gorenstein.

(ii) If $\rank(L \setminus I) = 1$, then $J_{L\setminus I}$ is the edge ideal of the complete graph $K_{\rho_1}$, which is Gorenstein if and only if $\rho_1\leq 2$ (\cite[Corollary 9.3.3]{HHgtm260}).

(iii)  Let $\rank(L \setminus I) =2$.  If $\rho_1=1$, then $J_{L\setminus I}$ is the edge ideal of the complete graph $K_{\rho_2}$. So, it is Gorenstein if and only if $\rho_2\leq 2$.

  Suppose that $\rho_1\geq 2$. If $(\rho_1, \rho_2)=(2,1)$, then $J_{L\setminus I}$ is a principal ideal and $\Rc_K[L\setminus I]$ is Gorenstein.  If $(\rho_1, \rho_2)=(3,3)$, then $\Rc_K[L\setminus I]=\Rc_K[B_3\setminus \hat{1}_{B_3}]$, which is Gorenstein (Theorem \ref{boolgor}).

Now, suppose that $\Rc_K[L\setminus I]$ is Gorenstein with $\rho_1\geq 2$.  Let $\Delta$ denote the order complex of $(L\setminus I)\setminus\{\hat{0}_L\}$ and $h(\Delta)=(h_0, h_1, h_2)$ its $h$-vector.  Let $H$ denote the Hasse diagram of $(L\setminus I)\setminus\{\hat{0}_L\}$, which is considered as a connected bipartite graph in the obvious way.  We distinguish the following three cases.

\medskip
{\bf Case 1.} Suppose that $H$ has at least two cycles.  Then $|E(H)|>|V(H)|$.  Hence,$$h_2=1-|V(H)|+|E(H)|\geq 2.$$Consequently, $\Rc_K[L\setminus I]$ cannot be Gorenstein.

\medskip
{\bf Case 2.} Suppose that $H$ has exactly one cycle.  If $\rho_1\geq 4$, then an interval of $L$ containing ${\hat 0}_L$ is the boolean lattice $B_4$.  Thus, $H$ has at least four cycles of length $6$, a contradiction.  Let $\rho_1\leq 3$. If $\rho_1=2$, as the two elements of rank $1$ have a unique join of rank $2$, $H$ can possess no cycle, a contradiction. Therefore, $\rho_1=3$.  It follows that $\rho_2\geq 3$ and the length of the unique cycle of $H$ is $6$. If $\rho_2=3$, then we are done.  Let $\rho_2\geq 4$ and $V(H)=\{x_1, \ldots, x_m\}$.  Set $p:=\pd{\Rc_K[L\setminus I]}$ and $r:=\reg{\Rc_K[L\setminus I]}$.  Since $\Rc_K[L\setminus I]$ is Cohen--Macaulay and $\dim\Rc_K[L\setminus I]=\rank(L\setminus I)+1=3$, we deduce from Auslander--Buchsbaum formula that $$p=(m+1)-(\rank(L\setminus I)+1)=m-2.$$ Since $\Delta$ is a connected simplicial complex of dimension one with $f$-vector $f(\Delta)=(|V(H)|, |E(H)|)$, one has $$\dim_k\widetilde{H}_1(\Delta; K)
=1-|V(H)|+|E(H)|=1.$$It follows from Hochster's formula \cite[Theorem 8.1.1]{HHgtm260} that $$\beta_{m-2,m}(K[L\setminus I]/I_{\Delta})=1.$$ Since $\Delta$ has $m$ vertices, $\beta_{m-2,s}(K[L\setminus I]/I_{\Delta})=0$ for each $s\geq m+1$. As $\beta_{p, p+r}((K[L\setminus I]/I_{\Delta})\neq 0$, we deduce that $r=2$. By using \cite[Corollary 2.7]{CV}, one has $$\beta_{m-2,m}(\Rc_K[L\setminus I])=\beta_{m-2,m}(K[L\setminus I]/I_{\Delta})=1.$$

Now, we study $\beta_{m-2,m-1}(\Rc_K[L\setminus I])$.  Let $C$ be the unique cycle of $H$ and choose $x_i\in V(H)\setminus V(C)$.  One has $\rank_{L}(x_i)=2$. Since $x_i\notin V(C)$, it follows that $x_i$ is a leaf of $H$.  Let $x_j$ be the unique neighbor of $x_i$.  one has $\rank_{L}(x_j) = 1$.  Hence, $x_ix_t\in J_{L\setminus I}$, for each $t=1, \ldots, m$ with $t\neq i,j$.  Considering the initial ideal of $J_{L\setminus I}$ with respect to $<_{\rm rev}$ introdeced in Section $2$, it follows that that$$(J_{L\setminus I}: x_i)=(x_t : t=1, \ldots m, t\neq i,j).$$The short exact sequence$$0\longrightarrow \frac{K[L\setminus I]}{(J_{L\setminus I}:x_i)}(-1)\longrightarrow \frac{K[L\setminus I]}{J_{L\setminus I}}\longrightarrow \frac{K[L\setminus I]}{(J_{L\setminus I},x_i)}\longrightarrow 0$$induces the exact sequence
\begin{align*}
\Tor_{m-1}^{K[L\setminus I]}\Big(K, \frac{K[L\setminus I]}{(J_{L\setminus I},x_i)}\Big)_{m-1} & \longrightarrow \Tor_{m-2}^{K[L\setminus I]}\Big(K, \frac{K[L\setminus I]}{(J_{L\setminus I}:x_i)}\Big)_{m-2}\\ & \longrightarrow \Tor_{m-2}^{K[L\setminus I]}\Big(K, \frac{K[L\setminus I]}{J_{L\setminus I}}\Big)_{m-1}.
\end{align*}
Since $J_{L\setminus I}$ is generated in degree two, one has$$\Tor_{m-1}^{K[L\setminus I]}\Big(K, \frac{K[L\setminus I]}{(J_{L\setminus I},x_i)}\Big)_{m-1}=0.$$Since $(J_{L\setminus I}:x_i)$ is generated by $m-2$ variables, one has $$\dim_K \Tor_{m-2}^{K[L\setminus I]}\Big(K, \frac{K[L\setminus I]}{(J_{L\setminus I}:x_i)}\Big)_{m-2}=1.$$The above exact sequence implies that $\beta_{m-2,m-1}(\Rc_K[L\setminus I])\geq 1$, which together with $\beta_{m-2,m}(\Rc_K[L\setminus I])=1$ yields $\beta_{m-2}(\Rc_K[L\setminus I])\geq 2$. Consequently, $\Rc_K[L\setminus I]$ cannot be Gorenstein, a contradiction.

\medskip
{\bf Case 3.} Suppose that $H$ is a tree.  We work with the same notations as in Case 2.  The argument done in Case 2 implies that $\widetilde{H}_1(\Delta; K)=0$. It follows from Hochster's formula that $\beta_{m-2,m}(K[L\setminus I]/I_{\Delta})=0$. Since $\Delta$ has $m$ vertices, $\beta_{m-2,s}(K[L\setminus I]/I_{\Delta})=0$, for each $s\geq m+1$. Since $\beta_{p, p+r}((\Rc_K[L\setminus I])=\beta_{p, p+r}((K[L\setminus I]/I_{\Delta})\neq 0$, we deduce that $r\leq 1$. Since $J_{L\setminus I}$ is generated in degree two, one has $r=1$. In other words, $J_{L\setminus I}$ has a linear resolution. Since $\Rc_K[L\setminus I]$ is Gorenstein, we deduce that $J_{L\setminus I}$ is a principal ideal.  Finally, it follows from $\rho_1\geq 2$ that $(\rho_1,\rho_2)=(2,1)$, as desired.
\end{proof}

A {\em level ring} introduced in Stanley \cite{Sta77} is a Cohen--Macaulay graded ring whose canonical module is generated in one degree.  Every Gorenstein ring is level.

\begin{Theorem}
    \label{level}
    Let $B_n$ be a boolean lattice and $I$ a rank-fixed dual order ideal of $B_n$.  Then $\Rc_K[B_n \setminus I]$ is level.
\end{Theorem}

\begin{proof}
Let $K[B_n \setminus I] = K[\{x_A : A \subset [n], A \in B_n \setminus I\}]$ and $\Delta$ the order complex of $B_n\setminus I$. As $I_{\Delta}$ is an initial ideal of $J_{B_n \setminus I}$, it is enough to prove that $K[B_n \setminus I]/I_{\Delta}$ is a level ring. Set $r:=\rank({B_n\setminus I})=\dim\Delta$. If $r=n$, then $I=\emptyset$ and $B_n\setminus I=B_n$ is Gorenstein. If $r=n-1$, then the assertion follows from Theorem \ref{boolgor}. So suppose that $r\leq n-2$. We know from Corollary \ref{regularity} that $r=\reg({K[B_n \setminus I]/I_{\Delta}})$. Set $p:=\pd({K[B_n \setminus I]/I_{\Delta}})$.

We show that $\beta_{p,p+\ell}(K[B_n \setminus I]/I_{\Delta})=0$ for each $\ell< r$. Let $I_0$ denote the rank-fixed dual order ideal of $B_n$ for which $\rank({B_n\setminus I_0})=r+1$ and $\Delta_0$ the order complex of $B_n\setminus I_0$. Let $A_1, \ldots, A_k$ denote the maximal elements of $B_n\setminus I_0$. In other words, $I\setminus I_0=\{A_1, \ldots, A_k\}$. One has $\pd({K[B_n \setminus I_0]/I_{\Delta_0}})=p+k-1$. In particular, $\beta_{p+k,p+k+\ell}(K[B_n \setminus I_0]/I_{\Delta_0})=0$ for each $\ell$. Set $J_0:=I_{\Delta_0}$ and $J_t:=I_{\Delta_0}+(x_{A_1}, \ldots, x_{A_t})$ for $1 \leq t \leq k$. For $t\leq k-1$ and $\ell< r$, the short exact sequence
$$0\longrightarrow \frac{K[B_n \setminus I_0]}{(J_t:x_{A_{t+1}})}(-1)\longrightarrow \frac{K[B_n \setminus I_0]}{J_t}\longrightarrow \frac{K[B_n \setminus I_0]}{J_{t+1}}\longrightarrow 0$$induces the exact sequence
\begin{align*}
\Tor_{p+k}^{K[B_n\setminus I_0]}\Big(K, \frac{K[B_n\setminus I_0]}{J_t}\Big)_{p+k+\ell} & \longrightarrow \Tor_{p+k}^{K[B_n\setminus I_0]}\Big(K, \frac{K[B_n\setminus I_0]}{J_{t+1}}\Big)_{p+k+\ell}\\ & \longrightarrow \Tor_{p+k-1}^{K[B_n\setminus I_0]}\Big(K, \frac{K[B_n\setminus I_0]}{(J_t:x_{A_{t+1}})}\Big)_{p+k+\ell-1}.
\end{align*}
The above exact sequence implies that
\[
\begin{array}{rl}
\beta_{p+k,p+k+\ell}\Big(\frac{K[B_n\setminus I_0]}{J_{t+1}}\Big)\leq \beta_{p+k,p+k+\ell}\Big(\frac{K[B_n\setminus I_0]}{J_t}\Big)+\beta_{p+k-1,p+k+\ell-1}\Big(\frac{K[B_n\setminus I_0]}{(J_t:x_{A_{t+1}})}\Big).
\end{array} \tag{1} \label{1}
\]
Note that$$(J_t: x_{A_{t+1}})=(I_{\Delta_0}:x_{A_{t+1}})+ ({\rm some \ variables})=I_{\link_{\Delta_0}({A_{t+1}})}+ ({\rm some \ variables}).$$Since $A_{t+1}$ is a maximal element of $B_n\setminus I_0$, we deduce that $I_{\link_{\Delta_0}{A_{t+1}}}$ is the Stanley--Reisner ideal of the order complex of $B_{|A_{t+1}|}=B_{r+1}$. Therefore, the regularity of $(J_t: x_{A_{t+1}})$ is $r$. On the other hand, by the same argument,$$\pd{\big(K[B_n\setminus I_0]/I_{\link_{\Delta_0}{(A_{t+1}})}\big)}=2^{|A_{t+1}|}-(|A_{t+1}|+1)=2^{r+1}-r-2.$$Since the variables belonging to $(J_t: x_{A_{t+1}})$ are all the variables $x_B\in K[B_n\setminus I_0]$ for which $B$ is not contained in $A_{t+1}$, we deduce that the number of these variables is $|B_n\setminus I_0|-2^{|A_{t+1}|}=|B_n\setminus I_0|-2^{r+1}$. Consequently,
\begin{align*}
\pd{\big(K[B_n\setminus I_0]/(J_t: x_{A_{t+1}})\big)} &=2^{r+1}-r-2+|B_n\setminus I_0|-2^{r+1}\\ &=|B_n\setminus I_0|-r-2=p+k-1.
\end{align*}
Since $K[B_n\setminus I_0]/I_{\link_{\Delta_0}{(A_{t+1}})}$ is Gorenstein, it follows that $K[B_n\setminus I_0]/(J_t: x_{A_{t+1}})$ is Gorenstein. In particular,$$\beta_{p+k-1,p+k+\ell-1}\Big(\frac{K[B_n\setminus I_0]}{(J_t:x_{A_{t+1}})}\Big)=0, \quad \ell< r.$$ We conclude from inequality (\ref{1}) that$$\beta_{p+k,p+k+\ell}\Big(\frac{K[B_n\setminus I_0]}{J_{t+1}}\Big)\leq\beta_{p+k,p+k+\ell}\Big(\frac{K[B_n\setminus I_0]}{J_t}\Big).$$A repeated application of the above inequality implies that$$\beta_{p+k,p+k+\ell}\Big(\frac{K[B_n\setminus I_0]}{J_k}\Big)\leq\beta_{p+k,p+k+\ell}\Big(\frac{K[B_n\setminus I_0]}{J_0}\Big).$$Recall that $J_0=I_{\Delta_0}$ and $\pd{(K[B_n \setminus I_0]/I_{\Delta_0})}=p+k-1$. Hence, the right hand side of the above inequality is zero. Thus $$\beta_{p+k,p+k+\ell}\Big(\frac{K[B_n\setminus I_0]}{J_k}\Big)=0,\quad \ell< r.$$
As $J_k=I_{\Delta_0}+(x_{A_1}, \ldots, x_{A_k})=I_{\Delta}+(x_{A_1}
, \ldots,  x_{A_k})$ and since the generators of $I_{\Delta}$ are note divisible by $x_{A_1}, \ldots, x_{A_k}$, we deduce from the above equality that $$\beta_{p,p+\ell}(K[B_n\setminus I]/I_{\Delta})=0.$$This shows that $K[B_n\setminus I]/I_{\Delta}$ is a level ring, as desired.
\, \, \, \, \, \, \, \, \, \, \, \, \, \, \,
\end{proof}

Theorem \ref{level} naturally yields the following conjecture.

\begin{Conjecture}
{\em
Let $\Delta$ be a triangulation of a sphere and $P_\Delta$ its {\em face poset}.  In other words, $P_\Delta$ consists of all faces of $\Delta$ including $\emptyset$, ordered by inclusion.  Recall that $P_\Delta$ is a simplicial poset  \cite{Sta91}.  Let $I$ be a rank-fixed dual order ideal of $P_\Delta$.  Then $\Rc_K[P_\Delta \setminus I]$ is level.
    }
\end{Conjecture}

It would, of course, be of interest to find a reasonable class of Cohen--Macaulay posets $P$ of distributive type for which $\Rc_K[P]$ is Gorenstein.

\section*{Acknowledgments}
The second author is supported by a FAPA grant from Universidad de los Andes.

\section*{Statements and Declarations}
The authors have no Conflict of interest to declare that are relevant to the content of this article.

\section*{Data availability}
Data sharing does not apply to this article as no new data were created or analyzed in this study.


\begin{thebibliography}{99}
\bibitem{Bac}
K.~Baclawski, Cohen--Macaulay ordered sets, {\em J. Algebra} {\bf 63} (1980), 226--258.
\bibitem{Bjo}
A.~Bj\"orner, Shellable and Cohen--Macaulay partially ordered sets, {\em Trans. Amer. Math. Soc.} {\bf 260} (1980), 159--183.
\bibitem{BjWa}
A.~Bj\"orner and M.~Wachs, Shellable nonpure complexes and posets II, {\em Trans. Amer. Math. Soc.} {\bf 349} (1997), 3945--3975.
\bibitem{BV}
B.~Benedetti and M.~Varbaro, On the dual graph of Cohen--Macaulay algebras, {\em Int. Math. Res. Not. IMRN} {\bf 17} (2015), 8085--8115.
\bibitem{CV}
A.~Conca and M.~Varbaro, Square-free Gröbner degenerations, {\em Inventiones Mathematicae} {\bf 221} (2020), 713--730.
\bibitem{Eis}
D.~Eisenbud, Introduction to algebras with straightening laws, {\em in} ``Ring Theory
and Algebra III'' (B.~R.~McDonald, Ed.), Proc. of the third Oklahoma Conf., Lect. Notes in Pure and Appl. Math. No. 55, Dekker, 1980, 243--268.
\bibitem{EHH}
V.~Ene, J.~Herzog and T.~Hibi, Linearly related polyominoes, {\em J. Algebraic Combinatorics} {\bf 41} (2015), 949--968.
\bibitem{HHgtm260}
J.~Herzog and T.~Hibi, ``Monomial Ideals'', GTM 260, Springer, 2011.
\bibitem{Hibi}
T.~Hibi, Distributive lattices, affine semigroup rings and algebras with straightening laws, {\it in} ``Commutative Algebra and Combinatorics''
(M. Nagata and H. Matsumura, Eds.), Advanced Studies in Pure Math.,
Volume 11, North--Holland, Amsterdam, 1987, pp. 93--109.
\bibitem{HIBIred}
T.~Hibi, ``Algebraic Combinatorics on Convex Polytopes,'' Carslaw Publications, Glebe, N.S.W., Australia, 1992.
\bibitem{OHH}
H.~Ohsugi, J.~Herzog and T.~Hibi, Combinatorial pure subrings, {\em Osaka J. Math.} {\bf 37} (2000), 745--757.
\bibitem{Rei}
G.~Reisner, Cohen--Macaulay quotients of polynomial rings, {\em Advances in Math.} {\bf 21} (1976), 30--49.
\bibitem{Sta75}
R.~P.~Stanley, The upper bound conjecture and Cohen--Macaulay rings, {\em Stud. Appl. Math.} {\bf 54} (1975), 135--142.
\bibitem{Sta77}
R.~P.~Stanley, Cohen--Macaulay complexes, ``Higher Combinatorics, Second Ed.'' (M.~Aigner, Ed.), NATO Advanced Study Institute Series, Reidel, Dordrecht/Boston, 1977, pp. 51--62.
\bibitem{StaGREEN}
R.~P.~Stanley, ``Combinatorics and Commutative Algebra, Second Ed.,'' Birh\"auser, Boston, 1996.
\bibitem{Sta91}
R.~P.~Stanley, $f$-vectors and $h$-vectors of
simplicial posets, {\em J. Pure and Appl. Algebra} {\bf 71} (1991), 319--331.
\bibitem{EC1}
R.~P.~Stanley, ``Enumerative Combinatorics, Volume I, Second Ed.,'' Cambridge Studies in Advanced Mathematics 49, Cambridge University Press, 2012.
\end{thebibliography}
\end{document}